\documentclass[titlepage,a4paper,11pt]{amsart} 
\usepackage[foot]{amsaddr}
\usepackage{amssymb}
\usepackage[lite]{amsrefs}
\usepackage{t1enc}
\usepackage[latin1]{inputenc}
\usepackage[english]{babel}
\usepackage{mathrsfs} 
\usepackage{hyperref}
\usepackage[all]{xy} 
\usepackage[lite]{amsrefs}
\usepackage[left=3cm,top=3cm,right=3cm,bottom=3cm]{geometry}

\newtheorem{thm}{Theorem}[section]
\newtheorem*{thm*}{Theorem}
\newtheorem{prop}[thm]{Proposition}
\newtheorem{lem}[thm]{Lemma}
\newtheorem{cor}[thm]{Corollary}
\newtheorem{defn}[thm]{Definition}
\newtheorem{rem}[thm]{Remark}
\newtheorem{ex}[thm]{Example}

\DeclareMathOperator\Pol{\mathcal{P}}
\DeclareMathOperator\Or{\mathcal{O}}
\DeclareMathOperator\W{\mathcal{W}}
\DeclareMathOperator\F{\mathcal{F}}

\DeclareMathOperator\g{\mathfrak{g}} 
\DeclareMathOperator\car{\mathfrak{t}}

\DeclareMathOperator\C{\mathbb{C}}
\DeclareMathOperator\R{\mathbb{R}}
\DeclareMathOperator\N{\mathbb{N}}

\DeclareMathOperator\SU{SU}
\DeclareMathOperator\U{U}

\DeclareMathOperator\dett{det}
\DeclareMathOperator\Tr{Tr}
\DeclareMathOperator\alt{alt}
\DeclareMathOperator\sym{sym}

\DeclareMathOperator\ev{ev}

\DeclareMathOperator\pr{pr}
\DeclareMathOperator\e{\epsilon}

\DeclareMathOperator\Ad{Ad}
\DeclareMathOperator\Id{Id}
\DeclareMathOperator\dimension{dim}

\hypersetup{
  colorlinks   = true, 
  urlcolor     = blue, 
  linkcolor    = blue, 
  citecolor   = red    
}
\title[On measures derived from orbital integrals]{On measures derived from orbital integrals}
\date{\today}
\author{Martin Miglioli}
\email[Martin Miglioli]{martin.miglioli@gmail.com}
\address[Martin Miglioli]{Instituto Argentino de Matem\'atica-CONICET. Saavedra 15, Piso 3, (1083) Buenos Aires, Argentina}
\thanks{The author was supported by IAM-CONICET, grants PIP 2010-0757 (CONICET) and PICT 2010-2478 (ANPCyT)}

\begin{document}

\begin{abstract}
The present work develops a framework to derive piecewise polynomial measures arising from invariant measures on adjoint orbits in the context of compact and semisimple Lie groups. These measures are computed from orbital integrals via transformations on spaces of polynomials endowed with the apolar inner product. In the case of the unitary group, we obtain a formula for the moments of the projection of an orbital measure.\\

\medskip

\end{abstract}

\maketitle




\section{Introduction}

We develop a novel approach to the derivation of some piecewise polynomial measures or Duistermat-Heckman type measures from Harish-Chandras orbital integral formula in the context of compact semisimple Lie groups. These measures are derived from orbital measures, that is, invariant probability measures on adjoint orbits of compact semisimple Lie groups. Specifically, we derive the pushforward of an orbital measure by the projection to the Cartan algebra, and we derive the radial part of the convolution of two orbital measures. In the case of type \textrm{A} root system we give a formula for the moments of the projection of an orbital measure. The characterization of these measures was previously addressed in \cite{zubov, olshanski, faraut2} and in \cite{faraut,cmz,zuber}. They are probabilistic versions of the Schur-Horn theorem \cite{horn1} and Horn's problem \cite{horn2}. Given hermitian matrices $X$ and $Y$ with fixed eigenvalues the Schur-Horn theorem characterizes the possible diagonal entries of $X$ as the convex hull of the permutation of eigenvalues, while Horn's problem asks for a characterization of the possible eigenvalues of the sum $X+Y$. The possible eigenvalues are given by a set of linear inequalities and Horn's problem was solved by Klyatchko in \cite{klyachko}.

The approach applies directly to compact semisimple Lie groups and is based on operations which can be done in finite dimensional spaces of polynomials endowed with the apolar inner product, which is the Segal-Bargmann-Fock space inner product restricted to polynomials. These operation are performed to both sides of an equation where the Fourier-Laplace transform of the unknown measure is written as an exponential polynomial involving the discriminant.  The framework in this article was motivated in part by \cite{miglioli} where the Harish-Chandra-Itzykson-Zuber (HCIZ) integral formula was put in the context of Segal-Bargmann spaces, and also \cite{stanley} where divided difference operators and their adjoints were used to solve algebraic problems. In the present work we considered finite-degree truncation of power series and performed operations on these polynomials. Proofs in which finite-degree terms are treated separately and an exponential function is constructed at the end are not uncommon in related literature, see for example the character expansion in the HCIZ integral in \cite{salamon} and the proof of the Duistermaat-Heckman localization formula \cite{salamon}.

The article is organized as follows. In Section \ref{secprel} we present the necessary results to state and prove the main theorems of the article, including several characterizations of the adjoint of the division by the discriminant. In Section \ref{sec1} we obtain a characterization of the pushforward of an orbital measure, and we obtain the moments of the measure for type \textrm{A} root systems. In Section \ref{sec2} we provide a characterization of the radial part of the convolution of two orbital measures. 

\section{Preliminaries}\label{secprel}
We review the results necessary to prove the main results of the article, such as the Harish-Chandra integral formula in Section \ref{orbital integrals} and operators between spaces of polynomials endowed with the apolar inner product in Section \ref{secspacespolyomials}. We also provide different characterizations of the adjoint of the division by the discriminant in Section \ref{secadjdiv}. In Section \ref{secadjdivA} we provide a characterization of this operator in the case of type \textrm{A} root systems. Most results are known, to the best of our knowledge Propositions \ref{padjtransl}, Proposition \ref{propajointinvderiv}, Theorem \ref{teocompdeI} and Proposition \ref{thmmatrixcoef} are new.

\subsection{Orbital measures and the Harish-Chandra integral formula}\label{orbital integrals}
Let $G$ be a compact, semisimple and connected Lie group with Lie algebra $\g$. We endow $\g$ with an $\Ad$-invariant inner product that is unique up to normalization and denote it by $\langle\cdot,\cdot\rangle$, $\W$ is the Weyl group, $\e(w)$ is the sign of $w\in\W$, that is, $\e(w)=(-1)^{\vert w\vert} $ where $\vert w\vert$ is the number of reflections necessary to generate $w$. A compact Lie group admits a unique invariant probability measure called the Haar measure and the adjoint action induces a unique invariant probability measure $\nu_a$ on the orbits $\Or_a=\Ad_G(a)$ for $a\in \g$. In \cite{harish} Harish-Chandra proved a formula for orbital integrals in Lie algebras in the case of compact, connected and semisimple Lie groups, see the expository article \cite{mcswiggen} and the references therein. The orbital integral admits an expression as an exponential polynomial. It is usually written

\begin{align*}
\Delta(x)\Delta(y)\int_G e^{\langle \Ad_gx,y\rangle}dg= \frac{\left[\Delta,\Delta\right]}{\vert\W\vert}\sum_{w\in\W}\epsilon(w)e^{\langle w(x),y\rangle},
\end{align*}
where $x$ and $y$ are in a Cartan algebra $\car_{\C}$ of $\g_{\C}$. The discriminant

$$\Delta(x)=\prod_{\alpha\in\Phi^+}\langle\alpha,x\rangle$$
is the product of the positive roots $\Phi^+$ considered as linear functionals. The term $\left[\Delta,\Delta\right]$ is computed with the apolar inner product which we are going to introduce in the next section. We usually consider the element $a\in\car$ of the orbit $\Or_a$ as fixed and write when $\Delta(a)\neq 0$

\begin{align}\label{eqhc}
\F_a(x)=\int_G e^{\langle \Ad_ga,x\rangle}dg=\frac{\left[\Delta,\Delta\right]}{\Delta(a)\Delta(x)}  \frac{1}{\vert\W\vert}\sum_{w\in\W}\epsilon(w)e^{\langle w(a),x\rangle}
\end{align}
as the Fourier-Laplace transform of the probability measure $\nu_a$ on the orbit $\Or_a$.

In the case of type \textrm{A} root systems we consider the unitary group $\U(n)$, its center is not trivial but this makes the link to the theory of symmetric polynomials more direct. The unitary orbital integral in this case is known as the Harish-Chandra-Itzykson-Zuber (HCIZ) integral \cite{itzykson}. It is usually written

\begin{equation}\label{formuHCIZ}
\begin{aligned}
\int_{\U(n)} e^{\Tr(uAu^{-1}B)}du&=\left(\prod_{p=1}^{n-1}p!\right)\frac{\dett{\left[ e^{a_ib_j}\right] _{i,j=1}^n}}{\Delta(A)\Delta(B)}\\
&=\left(\prod_{p=1}^{n}p!\right)\frac{\frac{1}{n!}\sum_{w\in S_n}e^{\langle A,w(B)\rangle}}{\Delta(A)\Delta(B)}
\end{aligned}
\end{equation}
where $\dett$ is the determinant of a matrix, $\U(n)$ is the group of $n$-by-$n$ unitary matrices, $A$ and $B$ are fixed $n$-by-$n$ diagonal matrices with eigenvalues $a_1< \dots <a_n$ and $b_1< \dots <b_n$ respectively, and
$$\Delta(A)=\prod_{i<j}(a_j-a_i)$$
is the Vandermonde determinant. We note that in this context 
$$\left[\Delta,\Delta\right]=\Delta(\partial)\Delta\vert_{z=0}=\prod_{p=1}^{n}p!,$$
where $\Delta(\partial)$ is the Vandermonde determinant with differentials $\frac{\partial}{\partial x_i}$ instead of the variables $x_i$.

\begin{rem}
The HCIZ has become an important identity in quantum field theory, random matrix theory, and algebraic combinatorics. 
\end{rem}

\subsection{Spaces of polynomials and operators between them}\label{secspacespolyomials}

Let $V$ be a vector space of dimension $n\in\N$ with the inner product $\langle \cdot, \cdot\rangle$. This space has coordinates $(x_1,\dots,x_n)$ given by an orthonormal basis. We define the space of polynomials on $V$ with degree not greater than $k\in \N_0$ as
$$\Pol^k(V)=\{f\in \Pol(V):\deg(f)\leq k\}.$$
This space is endowed with the apolar inner product 
$$\left[ f,g\right]=f(\partial)g(x)\vert_{x=0}$$
for $f,g\in \Pol^k(V)$. Here, $f(\partial)$ is the polynomial $f(\frac{\partial}{\partial x_1},\dots,\frac{\partial}{\partial x_n})$. This inner product is sometimes called Fisher or Bombieri inner product. If $(x_1,\dots,x_n)$ are coordinates given by an orthonormal basis of $V$ then the monomials given by $x^{\beta}/\sqrt{\beta !}$ are orthonormal, where $\beta\in\N^n$ is a multi-index and $\beta !:=\beta_1 !\beta_2 !\dots\beta_n!$.
Note that 
\begin{equation*}
	\left[ x^{\beta},x^{\gamma}\right] =\begin{cases}
	\beta! & \text{if $\beta=\gamma$}\\
	0  & \text{if $\beta\neq \gamma$}
	\end{cases}
\end{equation*}	
Let	
$$f(x)=\sum_{\beta}f_{\beta}x^{\beta}\qquad \mbox{and}\qquad g(x)=\sum_{\beta}g_{\beta}x^{\beta}$$ 
be two polynomials on $V$ where $\beta$ runs on multi-indices, then
$$\left[f,g\right]= f(\partial)g(x)\vert_{x=0}=\sum_{\beta}f_{\beta}g_{\beta}\beta !.$$
From this formula it follows that the apolar inner product is symmetric.

\begin{rem}
The apolar inner product is the Segal-Bargmann-Fock inner product \cite{hall,neretin} restricted to real polynomials. If $f,g$ are polynomials defined over the complex numbers then
$$g^*(\partial)f(z)\vert_{z=0}=\frac{1}{\pi^n}\int_{\C^n}f(z)\overline{g(z)}e^{-\vert z\vert^2}dz$$
where $g^*(z)=\overline{g(\overline{z})}$.
\end{rem} 

A multivariate version of Taylor's theorem applied to polynomials asserts that for a polynomial $f$ and $x,a\in V$
$$f(x+a)=e^{\langle a, \partial\rangle}f(x)=\sum_{\beta}\frac{a^{\beta}}{\beta !}\partial^{\beta}f(x).$$
If for $a\in V$ we denote by $T_af(x)=f(x+a)$ the translation operator on polynomials then $T_a=e^{\langle a, \partial\rangle}$ as operators on polynomials.
In the space $\Pol^k(V)$ we define for $a\in V$ the functions
$$q^k_a(x)=\sum_{l=0}^{k}\frac{1}{l!}\langle x,a\rangle^l$$
which are truncations of exponentials $e^{\langle x,a \rangle}$. We denote with $\ev_a(f)=f(a)$ the evaluation functionals at $a\in V$.

\begin{prop}\label{preprodq}
For $a\in V$ the functions $q_a^k$ has the reproducing property 
$$\left[f,q_a^k\right]=\ev_a(f)=f(a)$$
for $f\in\Pol^k(V).$
\end{prop}

\begin{proof}
For $a\in V$ and $f\in\Pol^k(V)$
\begin{align*}
\left[f,q_a^k \right]&=q_a^k(\partial)f(x)\vert_{x=0}=e^{\langle\partial,a\rangle}f(x)\vert_{x=0}\\
&=f(x+a)\vert_{x=0}=f(a)
\end{align*}
where we used the definition of the apolar inner product and Taylor's theorem for polynomials. This property also follows from the reproducing property of Segal-Bargmann spaces restricted to polynomials, see Section $5$ in \cite{hall}.
\end{proof}

Consider the case that the inner product space is endowed with a finite refection group $\W$ that has an alternating character $\epsilon:\W\to\{1,-1\}$. We can define the space of alternating polynomials 
$$\Pol^k_{\alt}(V)=\{f\in\Pol^k(V):f(w(x))=\epsilon(w)f(x)\mbox{  for all  }w\in\W\}$$
and the space of symmetric polynomials
$$\Pol^k_{\sym}(V)=\{f\in\Pol^k(V):f(w(x))=f(x)\mbox{  for all  }w\in\W\}.$$
In $\Pol^k(V)$ we define the orthogonal projection onto the alternating polynomials $P_{\alt}$ and the orthogonal projection onto the symmetric $P_{\sym}$, they are given by
$$P_{\alt}f(x)=\frac{1}{\vert\W\vert}\sum_{w\in\W}\epsilon(w)f(w(x))\qquad\mbox{and}\qquad P_{\sym}f(x)=\frac{1}{\vert\W\vert}\sum_{w\in\W}f(w(x)).$$

\begin{lem}\label{lemaproj}
For $f\in \Pol^k(V)$ and $g\in \Pol_{\alt}^k(V)$ 
$$\left[P_{\alt}f,g \right]=\left[f,g \right]$$
\end{lem}

\begin{proof}
This follows from $\left[P_{\alt}f,g \right]=\left[f,P_{\alt}g \right]=\left[f,g \right]$, since $P_{\alt}^*=P_{\alt}$ and projections fix the vectors in their range.
\end{proof}

We define for $a\in V$ the function 
$$r_a^k=P_{\alt}q_a^k.$$

\begin{prop}\label{preprodr}
For $a\in V$ the function $r_a^k$ has the reproducing property 
$$\left[r_a^k,g\right]=g(a)$$ 
in $\Pol_{\alt}^k(V)$.
\end{prop}

\begin{proof}
Note that
\begin{align*}
\left[r_a^k,g \right]&=\left[P_{\alt}q_a^k,g \right]=\left[q_a^k,P_{\alt}g \right]\\
&=\left[q_a^k,g \right]=g(a)
\end{align*}
where we used the definition of $r_a^k$, $P_{\alt}^*=P_{\alt}$ and the reproducing property of $q_a^k$ sated in Proposition \ref{preprodq}.
\end{proof}
Analogous properties hold in the space of symmetric polynomials $\Pol_{\sym}^k(V)$ but we are not going to use them. 

For $a\in V$ we have the translation operator 
$$T_a:\Pol^k(V)\to\Pol^k(V)\qquad \mbox{given by}\qquad T_af(x)=f(x+a).$$ 
This operator has an adjoint $T_a^*:\Pol^k(V)\to\Pol^k(V)$. To characterize the adjoint we use the notation $F^k$ to truncate polynomials, if we have a polynomial $f(x)=\sum_{\beta}f_{\beta}x^{\beta}$ then 
$$F^k(f(x))=\sum_{\beta:\vert\beta\vert\leq k}f_{\beta}x^{\beta}.$$ 
We also denote with $M_fg=f.g$ the multiplication of polynomials. Define the operator 
$$F^kM_{q_a^k}:\Pol^k(V)\to\Pol^k(V)$$
which consists of multiplying by $q_a^k$ and discarding in the result the terms of degree greater than $k$.

\begin{prop}\label{padjtransl}
For $a\in V$ and $k\in\N_0$ 
$$T_a^*=F^kM_{q_a^k}.$$
\end{prop}

\begin{proof}
For $f,g\in \Pol^k(V)$
\begin{align*}
\left[F^kM_{q_a^k}f,g\right]&=e^{\langle \partial,a\rangle}f(\partial)g(x)\vert_{x=0}\\
&=f(\partial)e^{\langle \partial,a\rangle}g(x)\vert_{x=0}\\
&=f(\partial)g(x+a)\vert_{x=0}\\
&=f(\partial)T_ag(x)\vert_{x=0}\\
&=\left[f,T_ag\right].
\end{align*}
An alternative proof can be given by checking the equation in monomials $f(x)=x^{\beta}$ and $g(x)=x^{\beta}$ using the binomial formula for several variables. Thus, for multi-indices $\beta\leq\gamma$ 
\begin{align*}
\left[x^{\beta},T_ax^{\gamma}\right]&=\left[x^{\beta},(x+a)^{\gamma}\right]=\left[x^{\beta},\sum_{\delta\leq\gamma}\binom{\gamma}{\delta}x^{\delta}a^{\beta-\delta}\right]\\
&=\left[x^{\beta},\binom{\gamma}{\beta}x^{\beta}a^{\gamma-\beta}\right]=a^{\gamma-\beta}\beta!\frac{\gamma!}{\beta!(\gamma-\beta)!}.
\end{align*}
One the other hand, with a slight abuse of notation when using the not truncated exponential
$$\left[e^{\langle a,x\rangle}x^{\beta},x^{\gamma}\right]=\left[\sum_{\delta}\frac{1}{\delta!}a^{\delta}x^{\delta+\beta},x^{\gamma}\right]=\left[\frac{1}{(\gamma-\beta)!}a^{\gamma-\beta}x^{\gamma},x^{\gamma}\right]=a^{\gamma-\beta}\frac{\gamma!}{(\gamma-\beta)!}.$$
\end{proof}

\subsection{General characterization of the adjoint of the division by the discriminant}\label{secadjdiv} 
We now consider vector spaces which are the Cartan algebras $\car$ of Lie algebras of compact semisimple Lie groups. One of the essential properties of the discriminant $\Delta:\car\to\R$ is that it skew with respect to the action of $\W$, $\Delta(w(x))=\e(w)\Delta(x)$. This follows from the fact that if $\alpha$ is a simple root the reflection through the plane $\langle \alpha,x\rangle=0$ sends $\alpha\mapsto-\alpha$ and permutes the other positive roots. The next proposition is well known.

\begin{prop}\label{propdivdisc}
Every polynomial $f\in \Pol^k_{\alt}(\car)$ is divisible by $\Delta$.
\end{prop}
This proposition implies that
$$\Pol^{k+\vert\Phi^+\vert}_{\alt}(\car)=\Delta.\Pol^{k}_{\sym}(\car).$$

\begin{defn}\label{definicion}
The division by the discriminant is defined as
$$D_{\Delta}:\Pol^{k+\vert\Phi^+\vert}_{\alt}(\car)\to\Pol^{k}_{\sym}(\car)\qquad\mbox{given by}\qquad D_{\Delta}f(x)=\frac{f(x)}{\Delta(x)},$$
and its adjoint $I_{\Delta}=D_{\Delta}^*$ is
$$I_{\Delta}=D_{\Delta}^*:\Pol^{k}_{\sym}(\car)\to\Pol^{k+\vert\Phi^+\vert}_{\alt}(\car).$$
\end{defn}

We will often write $l=\vert\Phi^+\vert$ for brevity.  Next we give a computation in the case of polynomials of lowest degree. We denote with $\mathbf{1}$ the constant polynomial $\mathbf{1}(x)=1$.

\begin{prop}\label{pint1}
The operator $I_{\Delta}$ satisfies
$$I_{\Delta}\mathbf{1}=\frac{\Delta}{\left[\Delta,\Delta\right]}.$$
\end{prop}

\begin{proof}
We have $D_{\Delta}(\R\Delta)=\R\mathbf{1}$. The space $\R\Delta$ are the polynomials of minimum degree in $\Pol^k_{\alt}$ and they are orthogonal to any alternating polynomial $\Delta.f$ where $f$ is a symmetric polynomial without constant term. Similarly, the constants are orthogonal to any symmetric polynomial without constant term. It follows that $D_{\Delta}(\{\R\Delta\}^{\perp})=\{\R\mathbf{1}\}^{\perp}$. Therefore, $I_{\Delta}\mathbf{1}=D^*_{\Delta}\mathbf{1}=d\Delta$ for $d\in\R$. We have 
$$1=\left[D_{\Delta}\Delta,\mathbf{1}  \right]=\left[\Delta,I_{\Delta}\mathbf{1}\right]=\left[\Delta,d\Delta\right]=d\left[\Delta,\Delta\right].$$
Hence $d=\left[\Delta,\Delta\right]^{-1}$ and the conclusion of the proposition follows.
\end{proof}

The discriminant with differential variables or differential discriminant is the operator
$$\Delta(\partial)=\prod_{i=1}^l\langle\partial,\alpha_i\rangle=\prod_{i=1}^l\partial_{\alpha_i},$$
where we denote $\partial_{\alpha}=\langle \partial,\alpha\rangle$. 

\begin{lem}\label{lemweylfifferential}
The operator $\Delta(\partial)$ satisfies 
$$w\cdot \left(\Delta(\partial)\right)=\epsilon(w)\Delta(\partial)$$
for each $w\in\W$.
\end{lem}

\begin{proof}
For a $w\in\W$ and $\alpha\in\car$ note that $w\cdot\partial_{\alpha}=\partial_{w(\alpha)}$. Hence 
\begin{align*}
w\cdot \Delta(\partial)&=w\cdot\left( \prod_{i=1}^l\partial_{\alpha_i}\right)=\prod_{i=1}^l\partial_{w(\alpha_i)}\\
&=\epsilon(w)\prod_{i=1}^l\partial_{\alpha_i}=\epsilon(w)\Delta(\partial),
\end{align*}
where in the third inequality we used the definition of $\epsilon(w)$ a the parity of the number of positive roots which pass to negative roots under $w$.
\end{proof}

For $k\in\N_0$ we define the operator
$$ \Delta(\partial):\Pol^{k+\vert\Phi^+\vert}_{\alt}(\car)\to\Pol^{k}_{\sym}(\car).$$
Each directional derivative $\partial_{\alpha_i}$ decreases de degree of the polynomials by $1$ or maps to $0$, so $\Delta(\partial)$ maps polynomials of degree $k+l$, to polynomials of degree $k$. Also, by Lemma \ref{lemweylfifferential} for $f\in\Pol^{k+l}_{\alt}(\car)$ and $w\in\W$
$$w\cdot (\Delta(\partial)f)=(w\cdot\Delta(\partial))(w\cdot f)=\epsilon(w)\Delta(\partial)\epsilon(w)f=\Delta(\partial)f,$$
thus, $\Delta(\partial)$ maps into symmetric polynomials and it is well defined.

The adjoint $D_{\Delta}^*$ of $D_{\Delta}$ can be characterized as follows in the general case.

\begin{prop}\label{propajointinvderiv}
For $k\in\N_0$ the operator $D_{\Delta}^*$ is the inverse of the discriminant with differential variables $\Delta(\partial)$, that is, 
$$D_{\Delta}^*=\Delta(\partial)^{-1}.$$
\end{prop}

\begin{proof}
We use the fact that differentiation by a variable is the adjoint of multiplication by the same variable, therefore, multiplication by $\Delta$ is the adjoint of $\Delta(\partial)$.  For polynomials $f,g\in\Pol^k_{\sym}(\car)$
\begin{align*}
\left[ f,g\right]&=\left[ D_{\Delta}(\Delta.f),g\right]\\
&=\left[ \Delta.f,D_{\Delta}^*g\right]\\
&=\left[ f,\Delta(\partial)D_{\Delta}^*g\right],
\end{align*}
thus, $\Delta(\partial)D_{\Delta}^*=\Id$. Also, for polynomials $r,s\in\Pol^{k+l}_{\alt}(\car)$
\begin{align*}
\left[ r,s\right]&=\left[ \Delta.D_{\Delta}(r),s\right]\\
&=\left[ D_{\Delta}(r),\Delta(\partial)s\right]\\
&=\left[ r,D_{\Delta}^*\Delta(\partial)s\right],
\end{align*}
so that $D_{\Delta}^*\Delta(\partial)=\Id$.
\end{proof}

For $\alpha\in V$ where $V$ is a vector space with inner product we define the antiderivative operator
$$I_{\alpha}: \Pol(V)\to \Pol(V)$$
as
$$g(x)\mapsto I_{\alpha}g(x)=\int^{\langle x,\alpha^{\circ}\rangle}_0 g(x-\alpha t)dt,$$
where $\alpha^{\circ}=\alpha/\langle\alpha,\alpha\rangle$. Note that the integration is done on the segment from $x$ to the projection of $x$ to the hyperplane $\{\alpha\}^{\perp}$. For $\alpha$ such that $\|\alpha\|=1$, the signed distance of $x$ to the hyperplane $\{\alpha\}^{\perp}$ is $\langle x,\alpha\rangle$. 

\begin{prop}\label{propgradoIrightinverse}
For $\alpha\in V$ the operator $I_{\alpha}$ increases the degree by $1$, that is
$$\deg(I_{\alpha}f)=\deg(f)+1$$
for every $f\in\Pol(V)$. For every $f\in\Pol(V)$
$$\partial_{\alpha}I_{\alpha}f=f$$ 
holds, that is, $\partial_{\alpha}I_{\alpha}=\Id$. Also, the identity $I_{-\alpha}=-I_{-\alpha}$ holds.
\end{prop}

\begin{proof}
For a constant $c\neq 0$ and unit norm $\alpha\in V$ note that $\partial_{c\alpha}=c\partial_{\alpha}$. Also, by a linear change of variables in the definition of $I_{\alpha}$ we can assume that $\Vert\alpha\Vert=1$. Since an orthogonal change of coordinates leaves the equation unchanged we can prove the identity for the standard basis vector $\alpha=e_1$. We can also verify the identity for polynomials $x_1^kf_k(x_2,\dots,x_n)$ since polynomials of this type span the spaces of polynomials. Since the operators $I_{e_1}$ and $\partial_{e_1}$ don't act on the factors $f_k$ it is enough to verify the identity $I_{e_1}\partial_{e_1}=\Id$ in the monomials $x_1^k$. We have 
$$\partial_{e_1}(x_1^k)=\frac{1}{k}x_1^{k-1}\qquad \mbox{and} \qquad I_{e_1}(x_1^k)=\frac{1}{k+1}x_1^{k+1}.$$
It is easy to verify that $\partial_{e_1}I_{e_1}(x_1^l)=x_1^l$,that the operator $I_{e_1}$ increases the degrees by $1$ and that $I_{-e_1}=-I_{e_1}$.
\end{proof}

For each $k\in\N_0$ and each ordering $\alpha_1,\alpha_2,\dots\alpha_l$ of the positive roots we define the operator 
$$P_{\alt}I_{\alpha_1}I_{\alpha_2}\dots I_{\alpha_l}:\Pol^k_{\sym}(\car)\to \Pol^{k+\vert\Phi^+\vert}_{\alt}(\car).$$
This operator is linear, its range consists of alternating polynomials which have degrees less than or equal to $k+l$ since by Proposition \ref{propgradoIrightinverse} each antiderivative operator $I_{\alpha_i}$ increases the degree of polynomials by $1$. 

\begin{thm}\label{teocompdeI}
The operator $P_{\alt}I_{\alpha_1}I_{\alpha_2}\dots I_{\alpha_l}$ is the operator $D^*_{\Delta}$, that is,
$$I_{\Delta}=D_{\Delta}^*=\Delta(\partial)^{-1}=P_{\alt}I_{\alpha_1}I_{\alpha_2}\dots I_{\alpha_l}.$$
\end{thm}

\begin{proof}
We are going to prove that $P_{\alt}I_{\alpha_1}I_{\alpha_2}\dots I_{\alpha_l}$ is equal to $\Delta({\partial})^{-1}$. We show that $P_{\alt}I_{\alpha_1}I_{\alpha_2}\dots I_{\alpha_l}$ is a right inverse of the operator
$$\Delta(\partial):\Pol^{k+l}_{\alt}(\car)\to\Pol^k_{\sym}(\car).$$
Since $\Delta(\partial)$ is bijective from this it follows that $P_{\alt}I_{\alpha_1}I_{\alpha_2}\dots I_{\alpha_l}$ is a two-sided inverse of $\Delta(\partial)$. Therefore, we need to verify that 
$$\left(\prod_{i=1}^l\partial_{\alpha_i}\right)P_{\alt}I_{\alpha_1}I_{\alpha_2}\dots I_{\alpha_l}=\Id.$$
Consider an $f\in\Pol^k_{\sym}(\car)$ and evaluate 
\begin{align*}
\left(\prod_{i=1}^l\partial_{\alpha_i}\right)&P_{\alt}I_{\alpha_1}I_{\alpha_2}\dots I_{\alpha_l}f=\\&=\left(\prod_{i=1}^l\partial_{\alpha_i}\right)\frac{1}{\vert\W\vert}\sum_{w\in\W}\epsilon(w)(w\cdot(I_{\alpha_1}I_{\alpha_2}\dots I_{\alpha_l})f)\\
&=\left(\prod_{i=1}^l\partial_{\alpha_i}\right)\frac{1}{\vert\W\vert}\sum_{w\in\W}\epsilon(w)I_{w(\alpha_1)}I_{w(\alpha_2)}\dots I_{w(\alpha_l)}(w\cdot f)\\
&=\frac{1}{\vert\W\vert}\sum_{w\in\W}\epsilon(w)\left(\prod_{i=1}^l\partial_{\alpha_i}\right)I_{w(\alpha_1)}I_{w(\alpha_2)}\dots I_{w(\alpha_l)}(f)\\
\end{align*}
\begin{align*}
\phantom{\left(\prod_{i=1}^l\partial_{\alpha_i}\right)}&=\frac{1}{\vert\W\vert}\sum_{w\in\W}\epsilon(w)\left(\epsilon(w)\prod_{i=1}^l\partial_{w(\alpha_i)}\right)I_{w(\alpha_1)}I_{w(\alpha_2)}\dots I_{w(\alpha_l)}(f)\\
&=\frac{1}{\vert\W\vert}\sum_{w\in\W}\left(\prod_{i=1}^l\partial_{w(\alpha_i)}\right)I_{w(\alpha_1)}I_{w(\alpha_2)}\dots I_{w(\alpha_l)}(f)\\
&=\frac{1}{\vert\W\vert}\sum_{w\in\W}\partial_{w(\alpha_l)}\dots\partial_{w(\alpha_2)}\partial_{w(\alpha_1)}I_{w(\alpha_1)}I_{w(\alpha_2)}\dots I_{w(\alpha_l)}(f)\\
&=\frac{1}{\vert\W\vert}\sum_{w\in\W}f=f,
\end{align*}
where in the third equality we used that $f$ is symmetric, in the forth equality we used the property $w\cdot \Delta(\partial)=\epsilon(w)\Delta(\partial)$ of Lemma \ref{lemweylfifferential} and in the sixth equality we used the fact that $\partial_{\alpha}I_{\alpha}=\Id$ of Proposition \ref{propgradoIrightinverse}. 
\end{proof}

\begin{cor}\label{coropadjointdim1}
In the case of the one dimensional Cartan algebras $\car\simeq \R$ we have $I_{\Delta}=I_\alpha$ where $\alpha$ is the only positive root.
\end{cor}

\begin{proof}
This follows from Theorem \ref{teocompdeI} by noting that
$$I_{\Delta}=\frac{1}{2}(I_{\alpha}-I_{-\alpha})=\frac{1}{2}(I_{\alpha}+I_{\alpha})=I_{\alpha}.$$ 
We give an alternative proof, it is enough to verify the identity $D_{\Delta}^*=I_{\Delta}=I_{\alpha}$ in the monomials $x^k$ and $x^l$. We have 
$$D_{\Delta}(x^k)=\frac{1}{x}x^k=x^{k-1}\qquad\mbox{and}\qquad I_{\Delta}(x^l)=\int_0^{x}t^ldt=\frac{1}{l+1}x^{l+1}.$$
It is easy to verify that $\left[x^{k-1},x^l\right]=\left[x^k,\frac{1}{l+1}x^{l+1}\right]$. If $k\neq l+1$ then both side are zero and if $k=l+1$ then both sides are equal to $l!$. 
\end{proof}

The following values of $I_{\Delta}$ on the functions $P_{\sym}q^k_a$ hold. 

\begin{prop}\label{propcaractenexp}
For $k\in \N_0$ and $a\in \car$ such that $\Delta(a)\neq 0$ the identity 
$$I_{\Delta}(P_{\sym}q_a^{k})=\frac{1}{\Delta(a)} r_a^{k+\vert\Phi^+ \vert}$$
holds.
\end{prop}

\begin{proof}
Using the fact that $\langle\partial,\alpha\rangle e^{\langle x, a\rangle}=\langle \alpha, a\rangle e^{\langle x, a\rangle}$ we obtain the following well known identity for power series 
$$ \Delta(\partial)\left(\frac{1}{\vert\mathcal{W}\vert}\sum_{w\in\mathcal{W}}\epsilon(w)e^{\langle x, w(a)\rangle}\right)=\Delta(a)\sum_{w\in\mathcal{W}}e^{\langle x, w(a)\rangle}.$$
If we take polynomial parts by applying the truncation operator we get
$$\Delta(\partial)\left(r_a^{k+\vert\Phi^+\vert}\right)=\Delta(a)P_{\sym}q_a^{k}.$$ 
If we apply $\Delta(\partial)^{-1}$ on both sides in the case $\Delta(a)\neq 0$ we obtain the formula stated in the theorem.
\end{proof}

\subsection{Characterization of $D^*_{\Delta}$ in the case of type \textrm{A} root systems}\label{secadjdivA}
We compute the adjoint $I_{\Delta}=D_{\Delta}^*$ in the case of type \textrm{A} root systems by expressing $D_{\Delta}$ in bases of symmetric and antisymmetric polynomials and taking the transpose of this matrix. 

We start with a review of some facts about symmetric polynomials and their relation to the apolar inner product. Let $\N_0$ stand for the non negative integers. For $\mu\in \N_0^n$ we denote the monomials as usual with $x^{\mu}=x_1^{\mu_1}\dots x_n^{\mu_n}$ and we use the notations $\mu!=\mu_1!\dots\mu_n!$ and $\vert\mu\vert=\mu_1+\dots+\mu_n$. The set $\Pi$ of partitions is defined as
$$\Pi=\{\lambda\in \N_0^n:\lambda_1\geq\lambda_2\geq\dots\geq\lambda_n\}.$$
For $k\in\N_0$ we denote with
$$\Pi^k=\{\lambda\in \Pi:\vert\lambda\vert\leq k \}$$
the partitions of total size not greater than $k$.
We set 
$$\delta=(n-1,n-2,\dots,0).$$ 
For $\lambda\in\Pi$, if $(x_1,\dots,x_n)$ are the eigenvalues of a hermitian matrix $x$ we define 
$$\chi_{\lambda}(x) = s_{\lambda}(x_1,\dots,x_n),$$ 
where $s_{\lambda}$ is a Schur polynomial. These polynomials are defined by
$$s_{\lambda}(x_1,\dots,x_n)=\frac{a_{\lambda+\delta}(x_1,\dots,x_n)}{a_{\delta}(x_1,\dots,x_n)},$$
where 
$$a_{\mu}(x_1,\dots,x_n)=\dett\left[x_i^{\mu_j}\right]_{i,j=1}^n.$$
Note that $a_{\delta}(x_1,\dots,x_n)=\Delta(x_1,\dots,x_n)$ is the Vandermonde determinant.

\begin{rem}\label{remark}
The irreducible polynomial representations of the general linear group are labelled by Young diagrams, which we think of as vectors $\lambda \in\Pi$. The character of the $\lambda$-representation is given by $\chi_{\lambda}$, the Schur polynomial evaluated at the eigenvalues of an invertible matrix.
\end{rem}

For $\mu\in\Pi$ we denote with $m_{\mu}$ the monomial symmetric polynomial which are defined as the sum of $x^{\lambda}$ where $\lambda$ ranges over distinct permutations of $\mu$, thus, $m_{\mu}$ is the sum of $\vert\mathcal{W\cdot\mu}\vert=\vert S_n\cdot\mu\vert$ distinct monomials. Here $\vert S_n\cdot\mu\vert$ is the cardinality of the Weyl orbit of $\mu$. 

Schur polynomials can be expressed as linear combinations of monomial symmetric functions with non-negative integer coefficients $K_{\lambda\mu}$ called Kostka numbers
$$s_{\lambda }=\sum _{\mu }K_{\lambda \mu }m_{\mu },$$
see Proposition 4.4.3 in \cite{sagan}. The Kostka numbers $K_{\lambda\mu}$ are given by the number of semi-standard Young tableaux of shape $\lambda\in\Pi$ and weight $\mu\in\Pi$. References for Kostka numbers are Section VI.1 in \cite{mcdonald} and Theorem 2.11.2 in \cite{sagan} for the representation theoretic aspect. Therefore, noting that $\Delta=a_{\delta}$, we have for $\lambda\in\Pi$
\begin{align}\label{eqdeltakostka}
D_{\Delta}a_{\lambda+\delta}=\frac{a_{\lambda+\delta}}{a_{\delta}}=s_{\lambda}=\sum_{\mu}K_{\lambda\mu}m_{\mu}.
\end{align}
It is known that $(m_{\mu})_{\mu\in\Pi^k}$ and $(a_{\lambda+\delta})_{\delta\in\Pi^k}$ form algebraic bases of $\Pol^k_{\sym}(\R^n)$ and $\Pol^{k+l}_{\alt}(\R^n)$ respectively, where $l=\vert\Phi^+\vert=\frac{n(n-1)}{2}$.
Therefore
$$\left[a_{\lambda+\delta},a_{\lambda+\delta}\right]=n!(\lambda+\delta)!\qquad\mbox{and}\qquad \left[m_{\mu},m_{\mu}\right]=\vert S_n\cdot\mu\vert \mu!.$$
We scale these polynomials to be of unit norm and define for $\lambda\in\Pi^k$ 
$$a'_{\lambda+\delta}=\frac{1}{\sqrt{n!(\lambda+\delta)}}a_{\lambda+\delta}\qquad\mbox{and}\qquad m'_{\mu}=\frac{1}{\sqrt{\vert S_n \cdot \mu\vert \mu!}}m_{\mu}.$$
Thus, for $k\in \N$ the set $(a'_{\lambda+\delta})_{\lambda\in\Pi^k}$ is an orthonormal basis of $\Pol_{\alt}^{k+l}(\R^n)$ and the set $(m'_{\mu})_{\mu\in\Pi^k}$ is an orthonormal basis of $\Pol_{\sym}^{k}(\R^n)$. 

We obtain the following characterization of the operator $I_{\Delta}$ given by its matrix coefficients.

\begin{prop}\label{thmmatrixcoef}
For $k\in \N$ the operator $I_{\Delta}=D_{\Delta}^*:\Pol^{k}_{\sym}(\R^n)\to\Pol^{k+l}_{\alt}(\R^n)$ is given by
$$I_{\Delta}(m_{\mu})=\sum_{\lambda\in\Pi}\frac{\vert S_n \cdot \mu\vert \mu !}{n!(\lambda+\delta)!}K_{\lambda\mu}a_{\lambda+\delta}$$
for $\mu\in\Pi^k$.

\end{prop}

\begin{proof}
By equation (\ref{eqdeltakostka}) for $\lambda \in\Pi^k$
$$D_{\Delta}a_{\lambda+\delta}=\sum_{\mu}K_{\lambda\mu}m_{\mu}.$$
Therefore, 
$$D_{\Delta}a'_{\lambda+\delta}=\sum_{\mu}K'_{\lambda\mu}m'_{\mu},$$
with
$$K'_{\lambda\mu}=\frac{\sqrt{\vert S_n\cdot \mu\vert \mu!}}{\sqrt{n!(\lambda+\delta)!}}K_{\lambda\mu}.$$
Thus, for $\mu\in\Pi^k$
$$D^*_{\Delta}m'_{\mu}=\sum_{\lambda}K'_{\lambda\mu}a'_{\lambda+\mu},$$
and this implies that 
$$D^*_{\Delta}m_{\mu}=\sum_{\lambda}\frac{\vert S_n \cdot \mu\vert \mu !}{n!(\lambda+\delta)!}K_{\lambda\mu}a_{\lambda+\delta}.$$
\end{proof}

\section{The projection of an orbital measure to the Cartan algebra}\label{sec1}

For $a\in \car$ let $\nu_a$ be the orbital measure on $\Or_a$ and define the measure $\mu_a$ on the Cartan algebra as $\mu_a=\pr_*(\nu_a)$ where $\pr:\g\to\car$ is the orthogonal projection. The measure $\mu_a$ was characterized in \cite{zubov,olshanski,faraut2}, it can also be characterized as a Duistermaat-Heckman measure in the context of Hamiltonian torus actions, see Section $5$ in the book \cite{gls} and the references therein. In this section we obtain an alternative characterization and provide a formula for the moments in the case of type \textrm{A} root systems. 

\subsection{General characterization of the measure $\mu_a$.} We start with an equation satisfied by the measure $\mu_a$.

\begin{prop}\label{propprojmeasure}
For $a\in\car$ such that $\Delta(a)\neq 0$ the measure $\mu_a$ satisfies
\begin{align}\label{equationprojmeasure}
\frac{\left[\Delta,\Delta\right]}{\Delta(a)\Delta(x)}  \frac{1}{\vert\W\vert}\sum_{w\in\W}\epsilon(w)e^{\langle w(a),x\rangle}=\int_{\car} e^{\langle c,x\rangle}d\mu_a(x).
\end{align}
\end{prop}

\begin{proof}
Since $\langle \Ad_gx,y\rangle =\langle \pr(\Ad_gx),y \rangle$ for $y\in\car$, $x\in\g$ and $g\in G$ it is easy to very that the measure $\mu_a=\pr_*(\nu_a)$ satisfies 

$$\int_G e^{\langle \Ad_ga,x\rangle}dg=\int_G e^{\langle \pr(\Ad_ga),x\rangle}dg=\int_{\car} e^{\langle c,x\rangle}d\mu_a(x).$$
Thus, the orbital integral (\ref{eqhc}) implies that the  measure $\mu_a$ satisfies

$$\F_a(x)=\frac{\left[\Delta,\Delta\right]}{\Delta(a)\Delta(x)}  \frac{1}{\vert\W\vert}\sum_{w\in\W}\epsilon(w)e^{\langle w(a),x\rangle}=\int_{\car} e^{\langle c,x\rangle}d\mu_a(x).$$
\end{proof}

\begin{thm}\label{teocharact1}
For $a\in\car$ such that $\Delta(a)\neq 0$ the measure $\mu_a$ is characterized by 

\begin{align}\label{eqcaract1}
\int_{\car}f(c)d\mu_a(c)=\frac{\left[\Delta,\Delta\right]}{\Delta(a)}\ev_{a}\left(I_{\Delta}P_{\sym}f\right),
\end{align}
for every $f\in\Pol(\car)$.
\end{thm}

\begin{proof}
We write $\eta=\left[\Delta,\Delta\right]/\Delta(a)$ and $l=\vert\Phi^+\vert$. If we take the part of degree not greater than $k\in\N_0$ in $x$ of equation (\ref{equationprojmeasure}) of Proposition \ref{propprojmeasure} then we obtain

$$\eta\frac{r_a^{k+l}(x)}{\Delta(x)}=\int_{\car}q_c^k(x)d\mu_a(c).$$
If for a polynomial $f\in\Pol^k(\car)$ we take the inner product of both sides with $f_{\sym}=P_{\sym}f$ then in the LHS we get 

\begin{align*}
\left[ \eta D_{\Delta}r_a^{k+l} ,f_{\sym}\right]&=\eta.\left[ D_{\Delta}r_a^{k+l} ,f_{\sym}\right]\\
&=\eta.\left[ r_a^{k+l} ,I_{\Delta}f_{\sym}\right]\\
&=\eta.\left[ q_a^{k+l} ,I_{\Delta}f_{\sym}\right]\\
&=\eta.\ev_{a}\left(I_{\Delta}f_{\sym}\right)
\end{align*}
where we used Definition \ref{definicion} in the second equality, Lemma \ref{lemaproj} in the third and Proposition \ref{preprodr} in the final equality. On the other hand, in the LHS we get

\begin{align*}
\left[\int_{\car}q_c^kd\mu_a(c),f_{\sym}\right]
&=\int_{\car}\left[ q_c^k,f_{\sym}\right]d\mu_a(c)\\
&=\int_{\car}\ev_c(f_{\sym})d\mu_a(c)\\
&=\int_{\car}f_{\sym}(c)d\mu_a(c)\\
&=\int_{\car}f(c)d\mu_a(c),
\end{align*}
where we used Proposition \ref{preprodq} in the second equality and the symmetry of the measure $\mu_a$ in the final equality. Hence, the conclusion of the theorem holds.
\end{proof}

An alternative characterization of $I_{\Delta}=\Delta(\partial)^{-1}$ is provided by the next corollary.

\begin{cor}
The operator $I_{\Delta}=\Delta(\partial)^{-1}:\Pol_{\sym}^k(\car)\to \Pol_{\alt}^{k+l}(\car)$ is given by
\begin{align*}
\left(\Delta(\partial)^{-1}f\right)(x)=\frac{\Delta(x)}{\left[\Delta,\Delta\right]}\int_{\car}f(y)d\mu_x(y)
\end{align*}
for $f\in \Pol^k_{\sym}(\car)$.
\end{cor}

\begin{proof}
We can rewrite (\ref{eqcaract1}) as 
\begin{align*}
\left(I_{\Delta}f\right)(a)=\frac{\Delta(a)}{\left[\Delta,\Delta\right]}\int_{\car}f(c)d\mu_a(c)
\end{align*}
for every $f\in\Pol_{\sym}(\car).$ If we change the symbols $a$ and $c$ to $x$ and $y$ we obtain the equation stated in the corollary.
\end{proof}

\begin{ex}
Let us check that the measure $\mu_a$ characterized in this way is a probability measure. If we take as polynomial $f=\mathbf{1}$ then by Theorem \ref{teocharact1} and Proposition \ref{pint1}
$$\int_{\car}1d\mu(x)=\frac{\left[\Delta,\Delta\right]}{\Delta(a)}\ev_{a}\left(I_{\Delta}\mathbf{1}\right)=\frac{\left[\Delta,\Delta\right]}{\Delta(a)}\ev_{a}\left(\frac{\Delta}{\left[\Delta,\Delta\right]}\right)=\frac{\left[\Delta,\Delta\right]}{\Delta(a)}\frac{\Delta(a)}{\left[\Delta,\Delta\right]}=1.$$
\end{ex}

\begin{ex}
Consider the case in which $G=\SU(2)$. Using Corollary \ref{coropadjointdim1} and the identification $\car\simeq\R$ given by $(x,-x)\mapsto x$ 
\begin{align*}
\int_{\car}f(x)d\mu_a(x)&=\eta.\ev_{a}\left(I_{\Delta}P_{\sym}f\right)\\
&=\frac{1}{a}\int^{a}_0\frac{1}{2}\left( f(a- t)+f(-a+ t)\right)dt.
\end{align*}
It is easy to see that the density of the measure $\mu_a$ is given by the indicator function in the interval $\left[a,-a\right]$ of $\car$. Since adjoint orbits are spheres, that is, $\Or_a=\Ad_{\SU(2)}(a)\simeq S^2$ this is the fact known to Archimedes  that the projection of the surface measure of a sphere to an axis is given by an indicator function.  
\end{ex}

\begin{rem}
By Theorem \ref{teocompdeI} equation (\ref{eqcaract1}) can de written as
\begin{align*}
\int_{\car}f(x)d\mu_a(x)&=\frac{\left[\Delta,\Delta\right]}{\Delta(a)}\ev_{a}\left(I_{\Delta}P_{\sym}f\right)\\
&=\frac{\left[\Delta,\Delta\right]}{\Delta(a)}\ev_{a}\left(P_{\alt}I_{\alpha_1}I_{\alpha_2}\dots I_{\alpha_l}P_{\sym}f\right)\\
&=\frac{\left[\Delta,\Delta\right]}{\Delta(a)}\sum_{w\in\W}\epsilon(w)\left(I_{\alpha_1}I_{\alpha_2}\dots I_{\alpha_l}P_{\sym}f\right)(w^{-1}a).
\end{align*}
Each term is a constant times
$$\ev_{w^{-1}a}\left(I_{\alpha_1}I_{\alpha_2}\dots I_{\alpha_l}f\right),$$
for $f\in\Pol_{\sym}(\car)$. This is a linear functional over $\Pol_{\sym}(\car)$ which is given by integration with a piecewise polynomial measure. We give an informal description of this integration procedure. A term $I_{\alpha_1}I_{\alpha_2}\dots I_{\alpha_l}f(w^{-1}a)$ consists of a signed integration over the interval $[w^{-1}a,P_{\alpha_1^{\perp}}w^{-1}a]$ of a function which on each of its point $a_2$ is the signed integral over the interval $[a_2,P_{\alpha_2^{\perp}}a_2]$, and continuing in this way until we reach integrals of $f$ over intervals of the form $[a_l,P_{\alpha_l^{\perp}}a_l]$. More formally and in the case $\vert\Phi^+\vert=3$ we get 
\begin{align*}
I_{\alpha_1}I_{\alpha_2}I_{\alpha_3}f(w^{-1}a)&=\int_0^{\langle w^{-1}a,\alpha_1^{\circ}\rangle}I_{\alpha_2}I_{\alpha_3}f(w^{-1}a-\alpha_1t_1)dt_1\\
&=\int_0^{\langle w^{-1}a,\alpha_1^{\circ}\rangle}\int_0^{\langle w^{-1}a-\alpha_1t_1,\alpha_2^{\circ}\rangle}I_{\alpha_3}f(w^{-1}a-\alpha_1t_1-\alpha_2t_2)dt_2dt_1\\
&=\int_0^{\langle w^{-1}a,\alpha_1^{\circ}\rangle}\int_0^{\langle w^{-1}a-\alpha_1t_1,\alpha_2^{\circ}\rangle}\int_0^{\langle w^{-1}a-\alpha_1t_1-\alpha_2t_2,\alpha_3^{\circ}\rangle}\\
&\phantom{=}f(w^{-1}a-\alpha_1t_1-\alpha_2t_2-\alpha_3t_3)dt_3dt_2dt_1.
\end{align*}
Perhaps it would be convenient to consider measures on $\car$ which are pushforwards by linear maps of signed uniform measures on polytopes. Also, a simplifying root geometric framework would be necessary to handle the cancellations given by the operator $P_{\alt}$ and the symmetries of the function $f$.
\end{rem}

\begin{rem}
In Proposition \ref{propcaractenexp} we showed that
$$I_{\Delta}(P_{\sym}q_b^{k})=\frac{1}{\Delta(b)} r_b^{k+l}.$$
Therefore, if take $f=q_b^k$ in Theorem \ref{teocharact1} we get
\begin{align*}
\int_{\car}q_b^k (x)d\mu_a(x)&=\frac{\left[\Delta,\Delta\right]}{\Delta(a)}\ev_{a}\left(I_{\Delta}P_{\sym}q_b^k\right)\\
&=\frac{\left[\Delta,\Delta\right]}{\Delta(a)\Delta(b)}r_b^{k+l}(a)
\end{align*}
which was the main assumption of the theorem.
\end{rem}

\subsection{Moment characterization of $\mu_a$.} We provide a formula for the moments of the measure $\mu_a$ in the case of type \textrm{A} root systems. We will use the HCIZ integral (\ref{formuHCIZ}) presented in Section \ref{orbital integrals}.

\begin{thm}\label{teomomentos}
In the case $G=\U(n)$ the measure $\mu_b$ has the following characterization by moments. For a partition $\eta\in\Pi$   
$$\int_{\R^n}m_{\eta}(x)d\mu_b(x)=\sum_{\lambda\in\Pi}\frac{\delta!\vert S_n \cdot \eta\vert \eta !}{(\lambda+\delta)!}K_{\lambda\eta}s_{\lambda}(b),$$
and the moment of the monomial $x^{\eta}$ is
$$\int_{\R^n}x^{\eta}d\mu_b(x)=\sum_{\lambda\in\Pi}\frac{\delta!\eta !}{(\lambda+\delta)!}K_{\lambda\eta}s_{\lambda}(b).$$
\end{thm}

\begin{proof}
We have the identities  
$$\left[\Delta,\Delta\right]=\left(\prod_{p=1}^{n}p!\right)\qquad\mbox{and}\qquad \left(\prod_{p=1}^{n-1}p!\right)=\delta!,$$
where $\delta=(n-1,n-2,\dots,1)$. For a monomial symmetric polynomial $m_{\eta}$ Theorem \ref{teocharact1} and Proposition \ref{thmmatrixcoef} imply that  
\begin{align*}
\int_{\R^n}m_{\eta}(x)d\mu_b(x)&=\frac{\left[\Delta,\Delta\right]}{\Delta(b)}\ev_{b}\left(I_{\Delta}m_{\eta}\right)\\
&=\frac{\left[\Delta,\Delta\right]}{\Delta(b)}\ev_{b}\left(\sum_{\lambda}\frac{\vert S_n \cdot \eta\vert \eta !}{n!(\lambda+\delta)!}K_{\lambda\eta}a_{\lambda+\delta}\right)\\
&=\frac{\left[\Delta,\Delta\right]}{a_{\delta}(b)}\left(\sum_{\lambda}\frac{\vert S_n \cdot \eta\vert \eta !}{n!(\lambda+\delta)!}K_{\lambda\eta}a_{\lambda+\delta}(b)\right)\\
&=\left[\Delta,\Delta\right]\sum_{\lambda}\frac{\vert S_n \cdot \eta\vert \eta !}{n!(\lambda+\delta)!}K_{\lambda\eta}s_{\lambda}(b)\\
&=\sum_{\lambda}\frac{\delta!\vert S_n \cdot \eta\vert \eta !}{(\lambda+\delta)!}K_{\lambda\eta}s_{\lambda}(b),
\end{align*}
thus the first equation of the theorem holds. Also, it is easy to check that for a monomial $x^\eta$
$$P_{\sym}(x^{\eta})=\frac{1}{\vert S_n\cdot\eta\vert}m^{\eta}(x).$$
Hence, the second equation of the theorem follows.
\end{proof}
We denote $\overline{\mathrm{1}}=(1,1,\dots,1)$.

\begin{cor}\label{corodimensiones}
For a partition $\eta\in\Pi$ 
$$\sum_{\lambda\in\Pi}\frac{\delta!\eta !}{(\lambda+\delta)!}K_{\lambda\eta}\dimension(V_{\lambda})=1,$$
where $\dimension(V_{\lambda})$ is the dimension of representation of $\U(n)$ of weight $\lambda$. 
\end{cor}

\begin{proof}
We apply Theorem \ref{teomomentos} in the case $b=\overline{\mathrm{1}}$. In this case the measure $\mu_{\overline{\mathrm{1}}}$ is the Dirac mass at $\overline{\mathrm{1}}$. Also, by Remark \ref{remark} $s_{\lambda}(\overline{\mathrm{1}})$ is the dimension of the representation with highest weight $\lambda$ which is given by the trace of the identity operator in representation space.
\end{proof}

\begin{rem}
For $k\in\N$ we can apply Corollary \ref{corodimensiones} to the partition $\eta=(k)$ of the symmetric representation $V_{(k)}$. Note that $K_{\lambda\lambda}=1$ for all partitions $\lambda$ by Proposition 4.4.3 in \cite{sagan}. Also, there are no partitions which strictly dominate $(k)$, and $\dimension(V_{(k)})=\binom{n+k-1}{k}$. We see that  
$$\frac{\delta!(k)!}{((k)+\delta)!}K_{(k)(k)}\dimension(V_{(k)})=\frac{\delta!(k)!}{((k)+\delta)!}\binom{n+k-1}{k}=1.$$
\end{rem}

\section{The radial part of the convolution of orbital measures}\label{sec2}

Let $\nu_a$ and $\nu_b$ be the orbital measures on $\Or_a$ and $\Or_b$ for $a,b\in\car$. We characterize the radial part $\nu_{a,b}$ of the convolution of measures $\nu_a*\nu_b$. This was done in previous research in \cite{faraut,itzykson,cmz} using mainly Fourier analytic techniques. We start with an equation satisfied by the radial measure $\nu_{a,b}$.

\begin{prop}\label{propeqconvmeasure}
For $a,b\in\car$ such that $\Delta(a)\neq 0$ and $\Delta(b)\neq 0$ the radial measure $\nu_{a,b}$ satisfies 
\begin{align}\label{eqconvmeasure}
\frac{\left[\Delta,\Delta\right]}{\Delta(a)\Delta(b)}\frac{1}{\Delta(x)} & \left(\frac{1}{\vert\W\vert}\sum_{w\in\W}\epsilon(w)e^{\langle w(a),x\rangle}\right)\left( \frac{1}{\vert\W\vert}\sum_{w\in\W}\epsilon(w)e^{\langle w(b),x\rangle}\right)\\
&=\int_{\car}\frac{1}{\Delta(x)} \left(\frac{1}{\vert\W\vert}\sum_{w\in\W}\epsilon(w)e^{\langle w(x),x\rangle}\right) d\nu_{a,b}(x).\nonumber
\end{align}
\end{prop} 

\begin{proof}
By Proposition 2.1 in \cite{faraut} the radial measure $\nu_{a,b}$ of the convolution $\nu_a*\nu_b$ is given by the Fourier-Laplace transform as 
$$\F_a(z)\F_b(z)=\int_{\car}\F_c(z)d\nu_{a,b}(x).$$
If we express the Fourier-Laplace transform terms $\F_a,\F_b$ and $\F_c$ using the orbital integral formula (\ref{eqhc}), and if we multiply both sides by $\Delta(x)$ and divide both sides by $\left[\Delta,\Delta\right]$ then we obtain the equation (\ref{eqconvmeasure}) in the statement of the proposition. 
\end{proof}

The next lemma ensures that the main result of this section is well formulated.

\begin{lem} For an alternating polynomial $f\in\Pol_{\alt}$ the polynomial
$$g=\frac{1}{\vert\W\vert}\sum_{w\in\W}\epsilon(w)T_{w(a)}f$$
is symmetric.
\end{lem}

\begin{proof}
For $w'\in\W$ we compute
\begin{align*}
w'\cdot g&=\frac{1}{\vert\W\vert}\sum_{w\in\W}\epsilon(w)w'\cdot(T_{w(a)}f)=\frac{1}{\vert\W\vert}\sum_{w\in\W}\epsilon(w)(T_{w'w(a)}(w'\cdot f))\\
&=\frac{1}{\vert\W\vert}\sum_{w\in\W}\epsilon(w)(T_{w'w(a)}\epsilon(w') f)=\frac{1}{\vert\W\vert}\sum_{w\in\W}\epsilon(w'w)(T_{w'w(a)} f)\\
&=\frac{1}{\vert\W\vert}\sum_{w\in\W}\epsilon(w)(T_{w(a)} f)=g.
\end{align*}
\end{proof}

\begin{thm}\label{teocharact2}
For $a,b\in\car$ such that $\Delta(a)\neq 0$ and $\Delta(b)\neq 0$ the radial part $\nu_{a,b}$ of the convolution $\nu_a*\nu_b$ of orbital measures is characterized by

\begin{align}\label{eqcharact2}
\int_{\car}\frac{f_{\alt}(c)}{\Delta(c)}d\nu_{a,b}(c)=\frac{\left[\Delta,\Delta\right]}{\Delta(a)\Delta(b)}\ev_b \left(I_{\Delta}\left(\frac{1}{\vert\W\vert}\sum_{w\in\W}\epsilon(w)T_{w(a)}f_{\alt}\right)\right).
\end{align}
for every alternating polynomial $f_{\alt}\in\Pol_{alt}(\car)$.
\end{thm}

\begin{proof}
We set $\varphi={{\left[\Delta,\Delta\right]}}/{\Delta(a)\Delta(b)}$ and $l=\vert\Phi^+\vert$. If we take the part of degree not greater than $k\in\N_0$ of equation (\ref{eqconvmeasure}) we get
$$\varphi . F^kM_{r_a^k}\frac{r_b^{k+l}(x)}{\Delta(x)}=\int_{\car}\frac{1}{\Delta(c)}r_c^k(x)d\nu_{a,b}(c).$$
For an alternating polynomial $f_{\alt}\in\Pol_{\alt}^k(\car)$ we take the inner product of both sides with $f_{\alt}$. On the LHS we obtain

\begin{align*}
\left[ \varphi . F^kM_{r_a^k}D_{\Delta}r_b^{k+l} ,f_{\alt}\right]
&=\varphi .\left[\frac{1}{\vert\W\vert}\sum_{w\in\W}\epsilon(w) F^kM_{q_{w(a)}^k}D_{\Delta}r_b^{k+l} ,f_{\alt}\right]\\
&=\frac{1}{\vert\W\vert}\sum_{w\in\W}\varphi .\left[\epsilon(w) F^kM_{q_{w(a)}^k}D_{\Delta}r_b^{k+l} ,f_{\alt}\right]\\
&=\frac{1}{\vert\W\vert}\sum_{w\in\W}\varphi .\left[ D_{\Delta}r_b^{k+l} ,\epsilon(w)T_{w(a)}f_{\alt}\right]\\
&=\varphi .\left[ D_{\Delta}r_b^{k+l} ,\frac{1}{\vert\W\vert}\sum_{w\in\W}\epsilon(w)T_{w(a)}f_{\alt}\right]\\
&=\varphi .\left[ r_b^{k+l} ,I_{\Delta}\left(\frac{1}{\vert\W\vert}\sum_{w\in\W}\epsilon(w)T_{w(a)}f_{\alt}\right)\right]\\
&=\varphi .\left[ q_b^{k+l} ,I_{\Delta}\left(\frac{1}{\vert\W\vert}\sum_{w\in\W}\epsilon(w)T_{w(a)}f_{\alt}\right)\right]\\
&=\varphi .\ev_b \left(I_{\Delta}\left(\frac{1}{\vert\W\vert}\sum_{w\in\W}\epsilon(w)T_{w(a)}f_{\alt}\right)\right)
\end{align*}
where we used the definition $r_a^k=P_{\alt}q_a^k$ in the first equality, Proposition \ref{padjtransl} in the third, Proposition \ref{definicion} in the fifth, Lemma \ref{lemaproj} in the sixth and Proposition \ref{preprodq} in the final equality. On the other hand, on the RHS we get
\begin{align*}
\left[\int_{\car}\frac{1}{\Delta(c)}r_c^k(x)d\nu_{a,b}(c),f_{\alt}\right]&=\int_{\car}\frac{1}{\Delta(c)}\left[r_c^k(x),f_{\alt}\right]d\nu_{a,b}(c)\\
&=\int_{\car}\frac{1}{\Delta(c)}\ev_c(f_{\alt})d\nu_{a,b}(c)\\
&=\int_{\car}\frac{1}{\Delta(c)} f_{\alt}(c)d\nu_{a,b}(c),
\end{align*}
where we used Proposition \ref{preprodr}  in the second equality. Combining the result of taking the inner product on both sides we obtain the equation in the statement of the theorem. Since all symmetric polynomials are of the form $f_{\alt}/\Delta$ for an alternating polynomial $f_{\alt}$ and since the measure $\nu_{a,b}$ is symmetric we see that (\ref{eqcharact2}) completely characterizes the measure $\nu_{a,b}$.
\end{proof}
We can provide a slight variation of the characterization in Theorem \ref{teocharact2} of the radial measure $\nu_{a,b}$.

\begin{cor}
The radial part $\nu_{a,b}$ of the convolution $\nu_a*\nu_b$ is characterized by

\begin{align*}
\int_{\car}g(c)d\nu_{a,b}(c)=\frac{\left[\Delta,\Delta\right]}{\Delta(a)\Delta(b)}\ev_b \left(I_{\Delta}\left(\frac{1}{\vert\W\vert}\sum_{w\in\W}\epsilon(w)T_{w(a)}(\Delta. P_{\sym}g)f\right)\right)
\end{align*}
for every polynomial $g\in\Pol(\car)$.
\end{cor}

\begin{proof}
An alternating polynomial $f_{\alt}\in\Pol_{\alt}(\car)$ can be written as $f_{\alt}=\Delta. P_{\sym}g$ for a polynomial $g\in\Pol(\car)$. Making this substitution in formula (\ref{eqcharact2}) and noting that  
$$\int_{\car}(P_{\sym}h)(c)d\nu_{a,b}(c)=\int_{\car}h(c)d\nu_{a,b}(c)$$
for any polynomial $h\in\Pol(\car)$ the corollary follows.
\end{proof}

\begin{ex}
Consider the case in which $G=\SU(2)$. For $f_{\alt}\in\Pol_{\alt}(\car)$, using the identification $\car\simeq \R$ given by $(x,-x)\to x$ and Corollary \ref{coropadjointdim1} 

\begin{align*}
\int_{\car}f_{\alt}(x)\frac{1}{\Delta(x)} d\nu_{a,b}(x)&=\varphi .\ev_{b}\left(I_{\Delta}\left(\frac{1}{2}\left( T_af_{\alt}-T_{-a}f_{\alt}\right)\right)\right)\\
&=\frac{1}{2ab}\int^{b}_0 \frac{1}{2}\left(f_{\alt}(b+a- t)-f_{\alt}(b-a- t)\right)dt.
\end{align*}
One can verify that the density function $\phi$ of the measure $\nu_{a,b}$ is given when $0<b<a$ by 
\begin{equation*}
	\phi(x) =\begin{cases}
	\frac{x}{4ab} & \text{if $a-b\leq x \leq a+b$}\\
	\frac{-x}{4ab} & \text{if $-a-b\leq x \leq -a+b$}\\
	0  & \text{otherwise.}
	\end{cases}
\end{equation*}
\end{ex}

\section*{Acknowledgements}

I thank Pablo Zadunaisky for discussions and for suggesting the article \cite{stanley}.






\noindent
\end{document}